\documentclass{article}
\usepackage[left=1.3in,right=1.3in,bottom=1.3in,top=1.3in]{geometry}
\pdfoutput=1
\usepackage{inconsolata}

\usepackage{amsmath,amsfonts,amssymb}%
\usepackage{amsthm}%
\usepackage[usenames]{color}
\usepackage{setspace}
\usepackage{comment}
\usepackage[colorlinks,citecolor=darkblue,urlcolor=darkblue,linkcolor=darkgreen]{hyperref}
\usepackage{stmaryrd}
\usepackage[capitalize]{cleveref}
\usepackage{microtype}
\usepackage{wasysym}
\usepackage{textcomp}

\crefname{lemma}{Lemma}{Lemmas}
\crefname{corollary}{Corollary}{Corollaries}
\crefname{theorem}{Theorem}{Theorems}
\crefname{equation}{Equation}{Equations}

\makeatletter
\newcommand{\crefnames}[3]{%
\@for\next:=#1\do{%
\expandafter\crefname\expandafter{\next}{#2}{#3}%
}%
}
\makeatother
\crefformat{subsection}{\S#2#1#3}

\newtheorem{theorem}{Theorem}[section]%
\newtheorem{proposition}[theorem]{Proposition}%
\newtheorem{lemma}[theorem]{Lemma}%
\newtheorem{corollary}[theorem]{Corollary}%
\newtheorem{definition}[theorem]{Definition}%

\newenvironment{proofsketch}{%
	  \proof}{\endproof}

\newcommand{\diverges}{{\,\!\uparrow}}
\newcommand{\converges}{{\,\!\downarrow}}

\def\st{\,:\,}

\def\M{{\EM{\mathcal{M}}}}

\def\N{{\EM{\mathcal{N}}}}

\def\T{{\EM{\mathfrak{M}}}}
\def\Bmachine{{\EM{\mathfrak{B}}}}
\def\F{{\EM{\mc{F}}}}
\def\G{{\EM{\mc{G}}}}
\def\H{{\EM{\mc{H}}}}

\def\a{{\EM{\ol{c}}}}

\newcommand{\mfL}{\mathfrak{L}}

\newcommand{\leT}{\leq_{\mathrm{T}}}
\newcommand{\equivT}{\equiv_{\mathrm{T}}}
\newcommand{\pars}{\,\cdot\,}

\newcommand{\defn}[1]{{\bf{#1}}}
\newcommand{\defas}{{\EM{\ :=\ }}}

\def\neighbor{\EM{\mathbf{N}}}

\def\w{\EM{\omega}}

\def\Z{{\EM{{\mbb{Z}}}}}

\def\Nats{{\EM{{\mbb{N}}}}}

\def\Lang{\EM{\mathfrak{A}}}

\def\^{\EM{{}^{\And}}}

\def\:{\colon}

\def\And{\EM{\wedge}}

\def\<{\EM{\langle}}
\def\>{\EM{\rangle}}

\def\nl{\newline}

\def\EM#1{\ensuremath{#1}}
\def\mbb#1{\EM{\mathbb{#1}}}
\def\mbf#1{\EM{\mathbf{#1}}}
\def\mc#1{\EM{\mathcal{#1}}}

\def\ol#1{\EM{\overline{#1}}}

\def\ul#1{\underline{#1}}

\newcommand{\Powerset}{\ensuremath{\mathfrak{P}}}

\definecolor{MyGreen}{rgb}{.75,0,.75}
\definecolor{RealGreen}{rgb}{0,1,0}
\definecolor{DarkGreen}{rgb}{0.1,0.4,0.1}
\definecolor{ActualGreen}{rgb}{0.0,0.5,0.0}
\definecolor{MyBlue}{rgb}{0,0,1}
\definecolor{MyRed}{rgb}{1,0,0}

\definecolor{darkred}{rgb}{0.5,0,0}
\definecolor{darkgreen}{rgb}{0, 0.3,0}
\definecolor{darkblue}{rgb}{0,0,0.6}

\begin{document}

\title{On the computability of graph Turing machines}
\author{Nathanael Ackerman \\ 
\vspace*{-10pt}\ \\
\normalsize
Harvard University\\
\normalsize
Cambridge, MA 02138, USA\\
\normalsize
\texttt{nate@math.harvard.edu}
\and Cameron Freer\\
\vspace*{-10pt}\ \\
\normalsize
Remine\\
\normalsize
Falls Church, VA 22043, USA\\
\normalsize
\texttt{cameron@remine.com}
}
\date{}

\maketitle

\begin{abstract}
	We consider \emph{graph Turing machines}, a model of parallel computation 
	on a graph,
	in which each vertex is only capable of performing one of a finite number of operations. This model of computation is a natural generalization of several well-studied notions of computation, including ordinary Turing machines, cellular automata, and parallel graph dynamical systems.
We analyze the power of computations  that can take place in this model, both in terms of the degrees of computability of the functions that can be computed, and the time and space resources needed to carry out these computations.  We further show that 
	properties of the underlying graph have significant consequences for the power of computation thereby obtained. 
	In particular, we show that every 
	arithmetically definable set can be computed
	by a graph Turing machine in constant time, and that every computably enumerable Turing degree can be computed in constant time and linear space by a
	graph Turing machine whose underlying graph has finite degree.
\end{abstract}

\section{Introduction}
When studying large networks, it is important to understand what sorts of computations can be
performed in a distributed way on a given network.
In particular, it is natural to consider the 
setting
where each node acts independently in parallel, and where the network is specified separately from the computation to be performed.
In order to study networks whose size is considerably larger than can be held in memory by the computational unit at any single node, it is often useful to model the network as an infinite graph. (For a discussion of modeling large networks via infinite graphs, see, e.g., \cite{MR2555927}.)

We define a notion of \emph{graph Turing machine} that is meant to capture this setting. This notion generalizes several other well-known models of computation, including ordinary Turing machines, cellular automata,
and parallel graph dynamical systems. Each of these models, in turn, occurs straightforwardly as a special case of a graph Turing machine, suggesting that graph Turing machines capture a natural concept of parallel computation on graphs.

A graph Turing machine (henceforth abbreviated as ``graph machine'') performs computation on a 
vertex-labeled edge-colored directed multigraph satisfying certain properties. This notion of computation is 
designed to capture the idea that in each timestep, every vertex performs a limited amount of computation (in parallel, independently of the other vertices), and can only distinguish vertices connected to it when they are connected by different colors of edges.

In this paper we study the functions that can be computed using graph machines, which we call \emph{graph computable} functions.  As we will see, this parallel notion of computation will yield significantly greater computational strength than ordinary Turing machines, even when we impose constraints on the time and space resources allowed for the graph computation, or when we require the underlying graph to be efficiently computable.
We will see that the computational strength of graph machines is exactly that of
$\mbf{0}^{(\w)}$, the Turing degree of true arithmetic (thereby providing another natural construction of this degree).
We also examine the relationship between various properties of the underlying graph (e.g., finiteness of degree) and the computational strength of the resulting graph machines.

\subsection{Main results and overview of the paper}
We begin by introducing the notions of colored graphs, graph machines, and graph computability (including resource-bounded variants) in \cref{Graph Computing Definitions}.

Our main results fall into two classes: bounds on the computational power of arbitrary computable graph machines, and bounds among machines with an underlying graph every vertex of which has finite degree (which we say is of \emph{finite degree}).

Theorem~\ref{Total graph computable functions are computable from 0^w} states that every graph computable function is Turing reducible to $\mbf{0}^{(\w)}$.
In the other direction, we show in \cref{omega-jump} that this bound is attained by a single graph Turing machine. 

Sitting below $\mbf{0}^{(\w)}$ are the arithmetical Turing degrees, i.e., those less than $\mbf{0}^{(n)}$ 
for some $n \in \Nats$, where $\mbf{0}^{(n)}$ denotes the $n$-fold iterate of the halting problem.
We show in Corollary~\ref{Pointwise computable Turing degrees} that every arithmetical Turing degree contains a 
function that is
graph Turing computable, moreover in constant time.
(It remains open whether every degree below $\mbf{0}^{(\w)}$ can be achieved.)

We next 
show in Corollary~\ref{Finite degree implies computable in 0'} that 
functions determined by 
graph machines with underlying graph of finite degree
are reducible to
the halting problem, $\mbf{0}'$. 
Further, we show in Corollary~\ref{Constant degree implies computable} that if we restrict to graph 
machines
where every vertex has the same (finite) degree, then the resulting graph 
computable function is computable by an ordinary Turing machine. 

We also show in Theorem~\ref{Lower bound on finite degree main result} that every Turing degree below $\mbf{0}'$ is the degree of some linear-space graph computable function with underlying graph of finite degree. 
When the Turing degree is
$k$-computably enumerable (for some $k \in \Nats$) then we may further take the graph machine to run in constant time.

In \cref{Other graph-theoretic properties}, we sketch
two properties of the graph machine or its
underlying graph that do not restrict which functions are graph computable.
We show, in \cref{eff-comp}, that we may take the graph machine to be itself efficiently computable, without changing the degree of functions thereby computed. 
We also show, in \cref{symmetric-ok},
that the requirement that the graph be directed adds no generality: any function which can be computed by a graph machine can be computed by a graph machine whose underlying graph is symmetric. 

In 
\cref{representations-sec},
we examine how several other models of computation
relate to graph machines. 
A graph Turing machine can be thought of as a generalization of an ordinary Turing machine, where the one-dimensional read-write tape is replaced by an arbitrary graph. In \cref{otm-subsec}, we describe how to simulate an ordinary Turing machine via a graph Turing machine.
One of the fundamental difficulties when generalizing ordinary Turing machines to graph Turing machines is to figure out how to determine the location of the head. We have taken an approach whereby there is no unique head, and hence each node processes its own data. The approach taken by \cite{MR3163227} is to have a unique head, but to allow its location to be nondeterministic, in that it is allowed to move to any vertex connected to the current location that is displaying an appropriate symbol. (Note that they call their different notion a ``graph Turing machine'' as well.)

Our graph machines can also be viewed directly as dynamical systems, or as a sort of network computation. Cellular automata can be simulated by graph machines, as we show in 
\cref{cellular-subsec}, and parallel graph dynamical systems are essentially equivalent to the finite case of graph machines (see \cref{gds-subsec}).
Indeed, parallel graph dynamical systems include the above case of cellular automata, and also boolean networks and other notions of network computation, as described in \cite[\S2.2]{MR3332130}.

We conclude with \cref{Possible-extensions} on possible extensions of graph machines, and 
\cref{Open Questions} on open questions.

\subsection{Notation}
\label{Notation}
When $f\:A \to B$ is a partial function and $a \in A$, we let $f(a) \diverges$ signify that $f$ is not defined at $a$, and $f(a) \converges$ signify that $f(a)$ is defined at $a$. Suppose $f, g\:A \to B$ are partial functions. For $a \in A$ we say that $f(a) \cong g(a)$ if either ($f(a) \diverges$ and $g(a) \diverges$) or ($f(a)\converges$, $g(a) \converges$, and $f(a) = g(a)$). We say that $f \cong g$ when $(\forall a \in A)\ f(a) \cong g(a)$. If $f\:A \to \prod_{i \leq n}B_i$  and $k \leq n\in\Nats$, then we let $f_{[k]}\:A \to B_k$ be the composition of $f$ with the projection map onto the $k$'th coordinate. 

Fix an enumeration of computable partial functions, and for $e \in \Nats$, let $\{e\}$ be the $e$'th such function in this list. 
If $X$ and $Y$ are sets with $0 \in Y$, let $Y^{<X}= \{\eta\:X \to Y \st |\{a\st \eta(a) \neq 0\}| < \w\}$, i.e., the collection of functions from $X$ to $Y$ for which all but finitely many inputs yield $0$. 
(Note that by this notation we do \emph{not} mean partial functions from $X$ to $Y$ supported on fewer than $|X|$-many elements.)
For a set $X$, let $\Powerset_{<\w}(X)$ denote the collection of finite subsets of $X$. Note that the map which takes a subset of $X$ to its characteristic function is a bijection between $\Powerset_{<\w}(X)$ and $\{0,1\}^{<X}$. 

When working with computable graphs, 
sometimes the underlying set of the graph will be a finite coproduct of finite sets and $\Nats^k$ for $k\in\Nats$. The standard notion of computability for 
$\Nats$ transfers naturally to such settings, making implicit use of the
computable bijections between $\Nats^k$ and $\Nats$, and between $\coprod_{i \leq k} \Nats$ and $\Nats$, for $k \in \Nats$.
We will sometimes say \emph{computable set} to refer to some computable subset (with respect to these bijections) of such a finite coproduct $X$, and \emph{computable function} to refer to a computable function
having domain and codomain of that form or of the form $F^{<X}$ for some finite set $F$.

For sets $X, Y \subseteq \Nats$,
we write $X \leT Y$ when $X$ is Turing reducible to $Y$ (and similarly for functions and other computably presented countable objects). 

In several places we make use of (ordinary) Turing machines, described in terms of their state space, transition function, and alphabet of symbols.
For more details on results and notation in computability theory, see \cite{MR882921}.

\section{Graph computing}
\label{Graph Computing Definitions}

In this section we will make precise what we mean by graph Turing machines as well as graph computable functions.

\begin{definition}
\label{Colored graph}
A \defn{colored graph} is a tuple $\G$
of the form $(G, (L, V), (C, E), \gamma)$
where 
\begin{itemize}
\item $G$ is a set, called the \defn{underlying set} or the \defn{set of vertices},

\item $L$ is a set called the \defn{set of labels} and $V\: G \to L$ is the \defn{labeling function}. 

\item $C$ is a set called the \defn{set of colors} and $E\: G \times G \to \Powerset_{< \w}(C)$ is called the \defn{edge coloring}

\item $\gamma\: L \to \Powerset_{<\w}(C)$ is called the \defn{allowable colors function} and satisfies
\[
(\forall v, w \in G)\ E(v, w) \subseteq \gamma(V(v)) \cap \gamma(V(w)). 
\]
\end{itemize}

A \defn{computable colored graph} is a colored graph along with indices witnessing that $G$, $L$, and $C$ are computable sets and that $V$,$E$, and $\gamma$ are computable functions. 
\end{definition}

The intuition is that a colored graph is an edge-colored directed multigraph where each vertex is assigned a label, and such that among the edges between any two vertices, there is at most one edge of each color, and only finitely many colors appear (which must be among those given by $\gamma$ applied to the label). Eventually, we will allow each vertex to do some fixed finite amount of computation, and we will want vertices with the same label to perform the same computations. 

For the rest of the paper by a \emph{graph} we will always mean a colored graph, and will generally write the symbol $\G$ (possibly decorated) to refer to graphs. 

\begin{definition}
A graph $\G$ is \defn{finitary} when its set of labels is finite.
\end{definition}
For a finitary graph, without loss of generality one may further assume that the set of colors is finite and that every color is allowed for every label.

Notice that graphs, as we have defined them, are allowed to have infinitely many labels and edge colors, so long as the edges connecting to any particular vertex are assigned a finite number of colors, depending only on the label. However, there is little harm in the reader assuming that the graph is finitary.

	The main thing lost in such an assumption is the strength of various results providing upper bounds on the functions that can be computed using graph machines, as we take care to achieve our lower bounds (\cref{Pointwise computable Turing degrees}, 
	\cref{omega-jump}, and
\cref{Lower bound on finite degree main result})
	using finitary graphs.

Let $\G$ be a graph with underlying set $G$, and suppose $A \subseteq G$. Define $\G|_A$ to be the graph with underlying set $A$ having the same set of labels and set of colors as $\G$, such that the labeling function, edge coloring function, and allowable colors function of $\G|_A$ are the respective restrictions to $A$.

\begin{definition}
	A \defn{graph Turing machine}, or simply \defn{graph machine}, is
a tuple $\T = (\G, (\Lang, \{0,1\}), (S, s, \alpha), T)$
where 

\begin{itemize}
\item $\G = (G, (L, V), (C, E), \gamma)$ is a graph, called the \defn{underlying graph}. We will speak of the components of the underlying graph as if they were components of the graph machine itself.
	For example, we will call $G$ the \emph{underlying set} of $\T$ as well as of $\G$.

\item $\Lang$ is a finite set, called the \defn{alphabet}, having distinguished symbols $0$ and $1$.

\item
	$S$ is a countable set called the \defn{collection of states}.
\item
	$\alpha\:L \to \Powerset_{<\w}(S)$ is the \defn{state assignment function}.
\item
	$s$ is a distinguished state, called the \defn{initial state}, such that $s \in \alpha(\ell)$ for all
	$\ell \in L$. 

\item $T\: L \times \Powerset_{<\w}(C) \times \Lang \times S \to \Powerset_{<\w}(C) \times \Lang \times S$ is a function, called the \defn{lookup table}, such that for each $\ell \in L$ and each $z \in \Lang$,
\begin{itemize}
	\item if $c \not \subseteq \gamma(\ell)$ or $t \not \in \alpha(\ell)$, then $T(\ell, c, z, t) = (c, z, t)$, i.e., whenever the inputs to the lookup table are not compatible with the structure of the graph machine, the machine acts trivially, and

	\item if $t \in \alpha(\ell)$ and $c \subseteq \gamma(\ell)$, then $T_{[1]}(\ell, c, z, t) \subseteq \gamma(\ell)$
		and $T_{[3]}(\ell, c, z, t) \in \alpha(\ell)$, i.e.,
		whenever the inputs are compatible with the structure, so are the outputs. 
\end{itemize}
Further, for all $\ell \in L$, we have $T(\ell, \emptyset, 0, s) = (\emptyset, 0, s)$, i.e., if any vertex is in the initial state, currently displays $0$, and has received no pulses, then that vertex doesn't do anything in the next step.  This lookup table can be thought of as specifying a \emph{transition function}.
\end{itemize}

	A \defn{$\G$-Turing machine} (or simply a \defn{$\G$-machine})
	is a graph machine with underlying graph $\G$.

A \defn{computable graph machine} is a graph machine along with indices witnessing that $\G$ is a computable graph, $S$ is computable set, and $\alpha$ and $T$ are computable functions. 
\end{definition}

If $A$ is a subset of the underlying set of $\T$, then $\T|_A$ is the graph machine 
with underlying graph $\G|_A$ having the same alphabet, states, and lookup table as $\T$.

The intuition is that a graph machine should consist of a graph where at each timestep, every vertex is assigned a state and an element of the alphabet, which it displays. To figure out how these assignments are updated over time, we apply the transition function determined by the lookup table which tells us, given the label of a vertex, its current state, and its currently displayed symbol, along with the colored \emph{pulses} the vertex has most recently received, what state to set the vertex to, what symbol to display next, and what colored pulses to send to its neighbors.

For the rest of the paper, $\T = (\G, (\Lang, \{0,1\}), (S, s, \alpha), T)$ will denote a \emph{computable} graph machine whose underlying (computable) graph is $\G = (G, (L, V), (C, E), \gamma)$.

\begin{definition}
	A \defn{valid configuration} of $\T$ is a function $f\: G \to \Powerset_{<\w}(C) \times \Lang \times S$ such that for all $v \in G$, we have 
	\begin{itemize}
		\item $f_{[1]}(v) \subseteq \gamma(V(v))$ and
		\item $f_{[3]}(v) \in \alpha(V(v))$. 
	\end{itemize}
\end{definition}

In other words, a valid configuration is an assignment of colors, labels, and states
that is consistent with the underlying structure of the graph machine, in the sense that the pulses received and the underlying state at each vertex are consistent with what its allowable colors function and state assignment function permit.

\begin{definition}
A \defn{starting configuration} of $\T$ is a function $f\: G \to \Powerset_{<\w}(C) \times \Lang \times S$ such that
\begin{itemize}
	\item $(\forall v \in G)\ f_{[1]}(v) = \emptyset$, 
	\item $f_{[2]} \in \Lang^{< G}$, and
	\item $(\forall v \in G)\ f_{[3]}(v) = s$.
\end{itemize} 
We say that $f$	is \defn{supported} on a finite set $A\subseteq G$ if $f_{[2]}(v) = 0$ for all $v\in G\setminus A$.

Note that any starting configuration of $\T$ is always a valid configuration of $\T$. 
\end{definition}

In other words, a starting configuration 
is an assignment
in which
all vertices are in the initial state $s$, no pulses have been sent, and only finitely many vertices display a non-zero symbol.

Note that if $A$ is a subset of the underlying set of $\T$ and $f$ is a valid configuration for $\T$, then $f|_A$ is also a valid configuration for $\T|_A$. Similarly, if $f$ is a starting configuration for $\T$, then $f|_A$ is a starting configuration for $\T|_A$. 

\begin{definition}
\label{Run on a valid configuration}
Given a valid configuration $f$ for $\T$, the \defn{run} of $\T$ on $f$ is the function $\<\T, f\>\: G \times \Nats \to \Powerset_{<\w}(C) \times \Lang \times S$ satisfying, for all $v\in G$,
\begin{itemize}
\item $\<\T, f\>(v, 0) = f(v)$ and

\item 
	$\<\T, f\>(v, n+1) = T(V(v), X, z, t)$ 
	for all $n \in \Nats$,
		where
\begin{itemize}
	\item $z = \<\T, f\>_{[2]}(v, n)$ and $t = \<\T, f\>_{[3]}(v, n)$, and
	\item $X = \bigcup_{w \in G}
		\bigl(	E(w, v) \cap \<\T, f\>_{[1]}(w, n)\bigr)$.
\end{itemize} 
\end{itemize}

We say that a run \defn{halts at stage $n$} if 
	$\<\T, f\>(v, n) = \<\T, f\>(v, n+1)$ for all $v\in G$. 
\end{definition}

A run of a graph machine is the function which takes a valid configuration for the graph machine and a natural number $n$, and returns the result of letting the graph machine process the valid configuration for $n$-many timesteps. 

The following lemma is immediate from Definition~\ref{Run on a valid configuration}. 

\begin{lemma}
	\label{run on a valid configuration lemma}
Suppose $f$ is a valid configuration for $\T$. 
For $n\in\Nats$, define
$f_n \defas \<\T, f\>(\pars, n)$.
	Then the following hold.

\begin{itemize}
\item 
	For all $n \in \Nats$,  the function
	$f_n$ is a valid configuration for $\T$.

\item For all $n, m \in \Nats$ and $v \in G$,
\[
		f_{n+m}(v) = \<\T, f_n\>(v, m) = \<\T, f\>(v, n+ m).
\]
\end{itemize}
\end{lemma}

We now describe how a graph machine defines a function. 

\begin{definition}
	For $x \in \Lang^{<G}$, let $\widehat{x}$ be the valid configuration such that 
	$\widehat{x}(v) = (\emptyset, x(v), s)$
	for all $v \in G$.
Define 
\[
\{\T\}\:\Lang^{<G} \to \Lang^{G}
\]
to be the partial function such that 
\begin{itemize}
	\item $\{\T\}(x)\diverges$, i.e., is undefined, if the run $\<\T, \widehat{x}\>$ does not halt, and

	\item $\{\T\}(x) = y$ if $\<T, \widehat{x}\>$ halts at stage $n$ and
	for all  $v \in G$,
\[
		y(v) = \<\T, \widehat{x}\>_{[2]}(v, n).
\] 
\end{itemize}

Note that $\{\T\}(x)$
	is well defined as $\widehat{x}$ is always a starting configuration for $\T$. 
\end{definition}

While in general, the output of $\{\T\}(x)$ might have infinitely many non-zero elements, for purposes of considering which Turing degrees are graph computable, we will mainly be interested in the case where $\{\T\}(x) \in \Lang^{<G}$, i.e., when all but finitely many elements of $G$ take the value $0$.

When defining a function using a graph machine, it will often be convenient to have extra vertices whose labels don't affect the function being defined, but whose presence allows for a simpler definition. 
These extra vertices can be thought as ``scratch paper'' and play the role of extra tapes (beyond the main input/output tape) in a multi-tape Turing machine. We now make this precise. 

\begin{definition}
\label{(G X)-computable}
Let $X$ be an infinite computable subset of $G$.
A function $\zeta\: \Lang^{<X} \to \Lang^{X}$ is \defn{$\<\G, X\>$-computable} via $\T$ if 
\begin{itemize}
\item[(a)] $\{\T\}$ is total,

\item[(b)] for $x, y \in \Lang^{<G}$, if 
	$x|_{X} = y|_{X}$
	then $\{\T\}(x) = \{\T\}(y)$,
	and

\item[(c)] for all $x \in \Lang^{<G}$, for all $v \in G \setminus X$, we have $\{\T\}(x)(v) = 0$, i.e., when $\{\T\}(x)$ halts, $v$ displays $0$, and

\item[(d)] for all $x \in \Lang^{<G}$, we have $\{\T\}(x)|_{X} = \zeta(
	x|_{X})$.
\end{itemize}

	A function is \defn{$\G$-computable via $\T$} if it is
	$\<\G, X\>$-computable via $\T$
for some infinite computable $X\subseteq G$.
	A function is \defn{$\G$-computable} if it is
	$\G$-computable via $\T^\circ$ for some computable $\G$-machine $\T^\circ$.
A function is
\defn{graph Turing computable}, or simply \defn{graph computable}, when it is $\G^\circ$-computable for some computable graph $\G^\circ$.
\end{definition}

The following lemma captures the sense in which  functions that are 
$\<\G, X\>$-computable  via $\T$
are determined by their restrictions to $X$.

\begin{lemma}
	Let $X$ be an infinite computable subset of $\G$.
There is at most one function $\zeta\: \Lang^{<X} \to \Lang^{X}$ that is $\<\G, X\>$-computable via $\T$, and it must be Turing equivalent to $\{\T\}$.
\end{lemma}
\begin{proof}
	Suppose there is some $\zeta\: \Lang^{<X} \to \Lang^{X}$ that is $\<\G, X\>$-computable via $\T$.
	Then by Definition~\ref{(G X)-computable}(a), $\{\T\}$ is total. By Definition~\ref{(G X)-computable}(b), for any $x \in \Lang^{<G}$, the value of $\{\T\}(x)$ only depends on $x|_{X}$, and so $\{\T\}$ induces a function $\delta\: \Lang^{<X} \to \Lang^{G}$.
By Definition~\ref{(G X)-computable}(d), 
the map $\Lang^{<X} \to \Lang^{X}$ given by $a\mapsto \delta(a)|_{X}$ is the same as $\zeta$. Therefore there is at most one function $\Lang^{<X} \to \Lang^{X}$ that is $\<\G, X\>$-computable via $\T$. 

By Definition~\ref{(G X)-computable}(c), $\{\T\}(x)|_{G \setminus X}$ is the constant $0$ function for all $x$. 
Therefore $\{\T\}$ is Turing equivalent to $\zeta$.
\end{proof}

\subsection{Resource-bounded graph computation}

Just as one may consider
ordinary computability restricted by bounds on the time and space needed for the computation, one may devise and study complexity classes for graph computability.
However, as we will see in 
\S\S\ref{infinite-degree-lower-bounds} and \ref{finite-degree-lower-bounds}, unlike with ordinary computability, a great deal can be done with merely \emph{constant time},
and
in the finitary case, our key constructions can be carried out by machines that run in \emph{linear space} --- both of which define here.

	Throughout this subsection, $Q$ will be a collection of functions from $\Nats$ to $\Nats$.

\begin{definition}
	A function $\zeta$
	is \defn{$Q$-time computable via $\T$} if 
\begin{itemize}
\item[(a)] $\zeta$ is $\G$-computable via $\T$, and

\item[(b)] 
	there is a $q\in  Q$ such that
	for all finite connected subgraphs $A\subseteq G$ and all starting configurations $f$ of $\T$ 
supported on $A$,
\[
	(\forall v\in G) \ \ \<\T, f\>\bigl(v, q(|A|)\bigr) = \<\T, f\>\bigl(v, q(|A|)+1\bigr)
\]
i.e., $\T$ halts in at most $q(|A|)$-many timesteps.
\end{itemize}
A function is \defn{$Q$-time graph computable} if it is $Q$-time computable 
	via $\T^\circ$ for 
some 	
	graph machine $\T^\circ$.
\end{definition}

In this paper, we consider mainly time bounds where $Q$ is the collection of constant functions $\{\lambda x.n \st n\in\Nats\}$, in which case we speak of \emph{constant-time graph computability}.

When bounding the space used by a computation, we will consider graph computations that depend only on a ``small'' neighborhood of the input.

\begin{definition}
	\label{nneigh}
	For each $A \subseteq G$ and $n \in \Nats$, 
the \defn{$n$-neighborhood} of $A$, written
	$\neighbor_n(A)$, is 
defined by induction as follows. \nl\nl
\ul{Case $1$}: The $1$-neighborhood of $A$ is 
\[
\neighbor_1(A) \defas A \cup \{v \in G\st (\exists a \in A)\ E(v,a) \cup E(a,v) \neq \emptyset\}.
\]
\ul{Case $k+1$}: The $k+1$-neighborhood of $A$ is 
\[
\neighbor_{k+1}(A) = \neighbor_1(\neighbor_k(A)).
\]
\end{definition}

\begin{definition}
	A graph machine $\T$ \defn{runs in $Q$-space} if 
	$\{\T\}$ is total and 
	there are $p, q\in Q$ such that for any finite connected subgraph $A\subseteq G$ and any starting configuration $f$ of $\T$ that is supported on $A$, we have $|\neighbor_{p(n)}(A)| \leq q(n)$ and
	\[
\<\T, f\>(v, \pars) = 
			\<\T|_{\neighbor_{p(n)}(A)}, f|_{\neighbor_{p(n)}(A)}\> (v, \pars)\]
for all $v\in A$, 
		where $n \defas |A|$.

A function $\zeta$ is
	\defn{$Q$-space graph computable} via $\T$ if $\zeta$ is 
$\G$-computable 
	via $\T$ where $\T$ runs in $Q$-space.
	We say that $\zeta$ is  \defn{$Q$-space graph computable} if it is $Q$-space graph computable via $\T^\circ$ for some graph machine $\T^\circ$.
\end{definition}

In this paper, the main space bound we consider is where $Q$ is the collection of linear polynomials, yielding \emph{linear-space graph computability}.
This definition generalizes the standard notion of linear-space computation, and reduces to it in the case of a graph machine that straightforwardly encodes an ordinary Turing machine (for details of the encoding see \cref{otm-subsec}).
For such encodings, the only starting configurations yielding nontrivial computations are those supported on a neighborhood containing the starting location of the Turing machine read/write head.
In the case of arbitrary computable graph machines, computations can meaningfully occur from starting configurations supported on arbitrary connected subgraphs. This is why the bound on the size of the neighborhoods required to complete the computation is required to depend only on the size of a connected subgraph the starting configuration is supported on.

One way to view space-bounded graph computation is as computing functions that need 
only a finite amount of a graph to perform their computation, where this amount depends only on the ``size'' of the input (as measured by the size of a connected support of the starting configuration corresponding to the input). 
This perspective is especially natural if one thinks of the infinite graph as a finite graph that is larger than is needed for all the desired inputs.

\section{Arbitrary graphs}
\label{arbitrary-sec}

In this section, we consider the possible Turing degrees of total graph computable functions. We begin with 
a bound for finite graphs.

\begin{lemma}
	\label{lemma-finite-graph}
	Suppose $G$ is finite.
	Let $\mbf{h}$ be the map which takes a valid configuration $f$ for $\T$ and returns $n \in \Nats$ if $\<\T, f\>$ halts at stage $n$ (and not earlier), and returns $\infty$ if $\<\T, f\>$ doesn't halt. Then 
\begin{itemize}
\item $\<\T, f\>$ is computable and

\item $\mbf{h}$ is computable. 
\end{itemize}
\end{lemma}

\begin{proof}
	Because $G$ is finite,
	$\<\T, f\>$ is computable.
	Further, there are are only finitely many valid configurations of $\T$.
	Hence there must be some $n, k\in\Nats$ such that for all vertices $v$ in the underlying set of $\T$, we have 
	$\<\T, f\>(v, n) = \<\T, f\>(v, n+k)$, and the set of such pairs $(n,k)$ is computable.
	Note that 
	$\<\T, f\>$ halts
	if and only if
there is some $n$, less than or equal to the number of valid configuration for $\T$, for which this holds for $(n,1)$. 
Hence 
$\mbf{h}$, which searches for the least such $n$, is computable.
\end{proof}

We now investigate which Turing degrees are achieved by arbitrary computable graph machines.

\subsection{Upper bound}
We will now show that every graph computable function is computable from $\mbf{0}^{(\w)}$.

\begin{definition}
	Let $f$ be a valid configuration for $\T$, and let $A$ be a finite subset of $G$. We say that $(B_i)_{i \leq n}$ is an \defn{$n$-approximation of $\T$ and  $f$ on $A$} if
\begin{itemize}
\item $A  = B_0$, 

\item $B_i \subseteq B_{i+1} \subseteq G$
	for all $i < n$, and

\item if $B_{i+1} \subseteq B \subseteq G$  then for all $v\in B_{i+1}$,
\[
\<\T|_{B_{i+1}}, f_i|_{B_{i+1}}\>(v, 1) = \<\T|_{B}, f_i|_{B}\>(v, 1),
\]
\end{itemize}
where again $f_i \defas \<\T, f\>(\pars, i)$.
\end{definition}

The following proposition 
(in the case where $\ell = n - n'$)
states that if $(B_i)_{i \leq n}$ is an $n$-approximation of $\T$ and $f$ on $A$, 
then as long as we are only running $\T$ with starting configuration $f$ for $\ell$-many steps, and are only considering the states of elements within $B_{n'}$,
then it suffices to restrict $\T$ to $B_{n}$.

\begin{proposition}
\label{Approximation = Full}
	The following claim holds
for every $n \in \Nats$:
		For every valid configuration $f$ for $\T$, 
		and finite $A\subseteq G$, 
\begin{itemize}
	\item 
		there is an $n$-approximation of $\T$ and $f$ on $A$, and
	\item if $(B_i)_{i \leq n}$ is such an approximation, then 
		\begin{align*}
\label{Approximation = Full: Equation 1}
\notag
(\forall n' < n)
(\forall \ell \le n-n') (\forall v \in B_{n'})\hspace*{88pt} 
		\\
			\<\T|_{B_{n'+\ell}}, f|_{B_{n'+\ell}}\>(v, \ell)
= \<\T, f\>(v, \ell). \hspace*{10pt}(\square_n)
\end{align*}
\end{itemize}
\end{proposition}
\begin{proof}
We will prove this by induction on $n$.

	\vspace*{5pt}
\noindent \ul{Base case:} The claim is vacuous for $n=0$, as there are no nonnegative values of $n'$ to check.

	\vspace*{5pt}
\noindent \ul{Inductive case:} Proof of the claim for $n = k+1$, assuming its truth for $n\le k$.\nl
To establish $(\square_{k+1})$  consider
\begin{align*}
	(\forall v \in B_{n'})\ \<\T|_{B_{n' + \ell}}, f|_{B_{n' + \ell}}\>(v, \ell) 
	= \<\T, f\>(v, \ell)
\tag{$\dagger$}
\end{align*}
where 
	$n' < k+1$ and $\ell \le (k+1) - n'$.
	If $\ell < (k+1)-n'$ and $k=0$ then 
	$\ell = 0$, and
	$(\dagger)$ holds
	trivially. 
	If $\ell < (k+1)-n'$ and $k>0$ then 
	$\ell \le k -n'$, and so
	$(\dagger)$ holds by the inductive hypothesis $(\square_{k})$. 
	Hence we may restrict attention to the case where $\ell = (k+1)-n'$.

Let $f$ be a valid configuration for $\T$, let
$A \subseteq G$ be finite, and
let $(B_i)_{i \leq k}$ be a $k$-approximation of $\T$ and 
$f_1$
	on $A$.
Let $D_{k+1}$ be a subset of $G$ such that for every $v \in B_k$ and every color $c$, if vertex $v$ receives a pulse (in $\<\T, f\>$) 
of color $c$ 
at the start of timestep $1$,
then there is some vertex $d \in D_{k+1}$ which sends a pulse 
of color $c$ 
to $v$ 
during timestep $1$. Note that because there are only finitely many colors of pulses which elements of $B_k$ can receive, and because $B_k$ is finite, we can assume that $D_{k+1}$ is finite as well.

Now let $B_{k+1} = B_k \cup D_{k+1}$.
Because each vertex in $B_k$
	receives the same color pulses in $\T$ as it does in $\T|_{B}$ for any set $B$ containing $B_{k+1}$, we have that $\<\T|_{B}, f|_B\>(b, 1)$ agrees with $\<\T, f\>(b, 1)$ whenever $b \in B_k$ and $B_{k+1} \subseteq B$. Therefore $(B_i)_{i \leq k+1}$ is a $(k+1)$-approximation of $\T$ and $f$ on $A$, and $(B_k, B_{k+1})$ is a $1$-approximation of $\T$ and $f$ on $B_k$.

	If $n'=k$, then $\ell = 1$, and so $(\dagger)$ holds by $(\square_1)$ applied to the approximation
	$(B_k, B_{k+1})$.

If $n' < k$, then by induction, we may use $(\square_k)$ where the bound variable $f$ is instantiated by $f_1$, the bound variable $n'$ by $n'$, and the bound variable $\ell$ by $\ell - 1$, 
to deduce that
\begin{align*}
			\<\T|_{B_{n'+\ell - 1}}, f_1|_{B_{n'+\ell - 1}}\>(v, \ell - 1)
= \<\T, f_1\>(v, \ell - 1)
\end{align*}
for all $v\in B_{n'}$.

By Lemma~\ref{run on a valid configuration lemma}, 
we have
$\<\T, f_1\>(v,\ell - 1)= \<\T, f\>(v,\ell)$ 
and
$\<\T|_{B_{n'+\ell - 1}}, f_1|_{B_{n'+\ell - 1}}\>(v, \ell - 1) = 
\<\T|_{B_{n'+\ell - 1}}, f|_{B_{n'+\ell - 1}}\>(v, \ell)$
for all $v\in B_{n'}$.
Because
$(B_i)_{i \leq k}$ is a $k$-approximation of $\T$ and $f_1$ on $A$, 
we know that 
$\<\T|_{B_{n'+\ell - 1 }}, f|_{B_{n'+\ell - 1}}\>(v, \ell) = 
\<\T|_{B_{n'+\ell}}, f|_{B_{n'+\ell}}\>(v, \ell)$
for all $v\in B_{n'}$.

Therefore
$(\dagger)$ holds, and we have established
$(\square_{k+1})$. 
\end{proof}

We now analyze the computability of approximations and of runs.

\begin{proposition}
\label{Computability of approximations to graph machines}
Let $n\in\Nats$.
For all computable graph machines $\T$ and configurations $f$ that are valid for $\T$,
the following are $\mbf{f}^{(n)}$-computable, where
$\mbf{f}$ is the Turing degree of $f$. 
\begin{itemize} 
\item The collection $P_n(f) \defas \{(A, (B_i)_{i \leq n})\st A\subseteq G$ is finite and $(B_i)_{i \leq n}$ is an $n$-approximation of $\T$ and $f$ on $A\}$.

\item The function $f_n \defas \<\T, f\>(\pars, n)$.
\end{itemize}
Further, these computability claims are uniform in $n$.
\end{proposition}

\begin{proof}
	We will prove this by induction on $n$. The uniformity follows, since
	for all $n> 1$,
	we provide the same reduction to $\mbf{f}^{(n)}$ and parametrized earlier quantities.

	\vspace*{5pt}
\noindent \ul{Base case (a):} Proof of claim for $n = 0$.\nl
Let $\T$ be a graph machine and $f$ a valid configuration for $\T$.
Given a finite $A\subseteq G$, the sequence $(A)$ is the only $0$-approximation of $\T$ and $f$ on $A$. Hence $P_0(f) = \{(A,(A))\st A\subseteq G$ finite$\}$ is computable.
Further, $f_0 = f$
	is computable from $\mbf{f}^{(0)} = \mbf{f}$.

	\vspace*{5pt}
\noindent \ul{Base case (b):} 
Proof of claim for $n = 1$.\nl
Let $\T$ be a graph machine and $f$ a valid configuration for $\T$.
For each finite $A \subseteq G$ and each finite $B_0, B_1$ containing $A$ we can $\mbf{f}$-compute whether $\<\T|_{B_0}, f|_{B_0}\>(v, 1) = \<\T|_{B_1}, f|_{B_1}\>(v, 1)$, 
	and so we can 
$\mbf{f}'$-compute $P_1(f)$.
But by Proposition~\ref{Approximation = Full},
we know that if $(A, (A, B)) \in P_1$ then for any $v \in A$ we have $\<\T, f\>(v, 1) = \<\T|_B, f|_B\>(v, 1)$.  Hence we can compute 
$f_1$
	from $P_1$, and so it is $\mbf{f}'$-computable. 

	\vspace*{5pt}
\noindent\ul{Inductive case:} 
Proof of claim for $n = k+1$ (where $k\ge 1$), assuming it for $n=k$.\nl
Let $\T$ be a graph machine and $f$ a valid configuration for $\T$.
	We know that $(B_i)_{i \leq k+1}$ is a $(k+1)$-approximation of $\T$ and $f$ on $A$ if and only if both (i) the sequence $(B_i)_{i \leq k}$ is a $k$-approximation of $\T$ and $f_1$ on $A$, and (ii) the sequence $(B_k, B_{k+1})$ is a $1$-approximation of $\T$ and $f$ on $B_k$.

We can therefore compute $P_{k+1}(f)$ from 
	$P_k(f_1)$
	and $P_1(f)$. By the inductive hypothesis, $P_1(f)$ is $\mbf{f}'$-computable.  
Hence we must show that 
	$P_k(f_1)$ is
$\mbf{f}^{(k+1)}$-computable.
Also by the inductive hypothesis,
	$P_k(f_1)$ is computable
from the $k$'th Turing jump of 
	$f_1$, and $f_1$ 
	is $\mbf{f}'$-computable. 
Hence
	$P_k(f_1)$ is 
$\mbf{f}^{(k+1)}$-computable.

Finally, by 
Proposition~\ref{Approximation = Full},
if $(B_i)_{i \leq k+1}$ is an approximation of $\T$ for $f$ and $A$ up to $k+1$ then for any $v \in A$ we have $\<\T|_{B_{k+1}}, f|_{B_{k+1}}\>(v, k+1) = \<\T, f\>(v, k+1)$. We can therefore compute 
	$f_{k+1}$
	from $P_{k+1}(f)$ (which will find such an approximation). Hence 
	$f_{k+1}$
	is $\mbf{f}^{(k+1)}$-computable. 
\end{proof}

We then obtain the following two results.

\begin{corollary}
\label{Run is computable from omega th jump}
If	$f$ is a valid configuration for $\T$, 
	then 
$f_n$
	is $\mbf{f}^{(n)}$-computable and so $\<\T, f\>$ is $\mbf{f}^{(\w)}$-computable,
	where $\mbf{f}$ is the Turing degree of $f$.
\end{corollary}
\begin{proof}
	By Proposition~\ref{Approximation = Full}, for each $v \in G$ (the underlying set of $\T$) and each $n \in \Nats$, there is an approximation of $\T$ for $f$ and $\{v\}$ up to $n$. Further, by Proposition~\ref{Computability of approximations to graph machines} we can 
	$\mbf{f}^{(n)}$-compute such an approximation,
	uniformly in $v$ and $n$. 
	But if $(B_i^v)_{i \leq n}$ is an approximation of $\T$ for $f$ and $\{v\}$ up to $n$ then $\<\T|_{B^v_n}, f|_{B^v_n}\>(v, n) = \<\T, f\>(v, n)$. So $f_n = \<\T, f\>(\pars, n)$ is $\mbf{f}^{(n)}$-computable, uniformly in $n$. Hence $\<\T, f\>$ is $\mbf{f}^{(\w)}$-computable. 
\end{proof}

\begin{theorem}
\label{Total graph computable functions are computable from 0^w}
	Suppose that $\{\T\}\:\Lang^{<G} \to \Lang^{G}$ is a total function. Then $\{\T\}$ is computable from $\mbf{0}^{(\w)}$.
\end{theorem}
\begin{proof}
	Let $f$ be any starting configuration of $\T$. Then $f$ is computable.
Hence 
by Corollary~\ref{Run is computable from omega th jump},
$\<\T, f\>(v, n+1)$ is $\mbf{0}^{(n+1)}$-computable.
This then implies that the function determining
	whether or not $\{\T\}(x)$ halts after $n$ steps is $\mbf{0}^{(n+2)}$-computable.

	But by assumption, $\{\T\}(x)$ halts for every $x \in \Lang^{<G}$, and so $\{\T\}$ is $\mbf{0}^{(\w)}$-computable. 
\end{proof}

\subsection{Lower bound}
\label{infinite-degree-lower-bounds}

We have seen that every graph computable function is computable from $\mbf{0}^{(\w)}$. 
In this subsection, we will see that this bound can be obtained.
We begin by showing that every arithmetical 
Turing degree has an element that 
is graph computable in constant time. From this we then deduce
that there is a graph computable function Turing equivalent to $\mbf{0}^{(\w)}$. 

We first recall the following standard result from computability theory (see {\cite[III.3.3]{MR882921}}).

\begin{lemma}
\label{Description of arithmetical functions}
Suppose $n \in \Nats$ and 
$X\subseteq \Nats$. Then the following are equivalent.
\begin{itemize}
	\item $X \leT \mbf{0}^{(n)}$. 

\item There is a computable function $g\: \Nats^{n+1} \to \Nats$ such that 
\begin{itemize}
\item $h(\pars) \defas \lim_{x_0 \to \infty}\cdots \lim_{x_{n-1} \to \infty} g(x_0, \dots, x_{n-1}, \pars)$ is total. 

\item $h \equivT X$.
\end{itemize}
\end{itemize}
\end{lemma}

We now give the following construction.

\begin{proposition}
\label{Pointwise computable for iterative limits}
Let $n\in\Nats$ and suppose $g\: \Nats^{n+1} \to \Nats$ is computable such that 
\[
h(\pars) \defas \lim_{x_0 \to \infty}\cdots \lim_{x_{n-1} \to \infty} g(x_0, \dots, x_{n-1}, \pars)
\]
is total. Then $h$ is 
graph computable in constant time $5n+5$, via a graph machine whose labels, colors, alphabets, states, and lookup table are all finite and do not depend on $n$ or $g$.
\end{proposition}
\begin{proof}
The first step in the construction is to define a graph machine which can take the limit of a sequence. We will think of this a subroutine which we can (and will) call several times. Let $\G_{\mfL}$ be the following graph. 
\begin{itemize}
\item The underlying set is $\Nats \cup \{*\}$, where $*$ is some new element. 

\item There is only one label, $p$, which all vertices have.

\item The colors of $\G_\mfL$ are $\{\mbf{B}_0, \mbf{B}_1, \mbf{SB}, \mbf{SF}_0, \mbf{SF}_1, \mbf{A}\}$.

\item The edge coloring is $E_{\mfL}$, satisfying the following for all $m,m' \in\Nats$.

\begin{itemize}
\item 
	$E_{\mfL}(m, m') = \emptyset$ and $E_{\mfL}(m', m) = \{\mbf{B}_0, \mbf{B}_1\}$
	when $m < m'$. 

\item $E_{\mfL}(m, m) = \{\mbf{B}_0, \mbf{B}_1\}$. 

\item $E_{\mfL}(*, m) = \{\mbf{SB}\}$ and $E(m, *) = \{\mbf{SF}_0, \mbf{SF}_1\}$. 

\item $E_{\mfL}(*, *) = \emptyset$.
\end{itemize}
\end{itemize}

Let $\T_\mfL$ be the following graph machine.
\begin{itemize}
\item The underlying graph is $\G_\mfL$.

\item The alphabet is $\{0,1\}$.

\item The states of $\T_\mfL$ are $\{s, a_0, a_1, a_2, u, b\}$.

\item The lookup table $T_{\mfL}$ satisfies the following, for all 
	$z$ in the alphabet, states $x$, and collections $X$ of colors. 
\begin{itemize}
\item[(i)] $T_{\mfL}(p, \emptyset, z, s) = T_{\mfL}(\emptyset, 0, s)$. 

\item[(ii)] $T_{\mfL}(p, X \cup \{\mbf{A}\}, z, x) = (\emptyset, 0, a_0)$.

\item[(iii)] $T_{\mfL}(p, X, z, a_0) = (\{\mbf{SB}\}, 0, a_1)$.

\item[(iv)] $T_{\mfL}(p, X, z, a_1) = (\emptyset, 0, a_2)$.

\item[(v)] $T_{\mfL}(p, X, z, a_2) = (\emptyset, k, u)$ if $\mbf{SF}_{k} \in X$ and $\mbf{SF}_{1-k} \not \in X$ for some $k \in \{0,1\}$.
\item[(vi)] $T_{\mfL}(p, X \cup \{\mbf{SB}\}, z, x) = (\{\mbf{B}_{z}\}, 0, b)$ if $x \not \in \{a_0, a_1, a_2\}$.

\item[(vii)] $T_{\mfL}(p, X, z, b) = (\{\mbf{SF}_{k}\}, 0, u)$ if $\mbf{B}_{k} \in X$ and $\mbf{B}_{1-k} \not \in X$ for some $k\in\{0,1\}$.

\item[(viii)] $T_{\mfL}(p, X, z, b) = (\emptyset, 0, u)$ if $\{\mbf{B}_0, \mbf{B}_1\} \subseteq X$.

\item[(ix)] $T_{\mfL}(p, X, z, u) = (\emptyset, z, u)$ if $\mbf{A} \not\in X$. 
\item[(x)] $T_{\mfL}(p, X, z, x) = (\emptyset, 0, u)$ in all other cases.
\end{itemize}
\end{itemize}

We now describe what this graph machine does, beginning from a starting configuration. First, condition (i) sets everything to a clean slate, i.e., makes sure that at the beginning of the second timestep, every vertex will display $0$.
This ensures that the outcome won't depend on the values initially displayed on any vertex. 

Next, by condition (ii), if a vertex receives a pulse of type $\mbf{A}$ then at the next timestep it enters state $a_0$.
We can think of this as signaling that this subroutine 
has been ``activated''. This will only ever be sent to the element $*$, which we call the ``activation vertex''. 

Then, by conditions (iii) and (iv), once the activation vertex is in state $a_0$ it will send an $\mbf{SB}$-pulse to every other vertex. This signals to them to start calculating the limit. The activation vertex will then pause and enter state $a_1$ and then $a_2$. This pause will give the other vertices an opportunity to calculate their limiting value. 

The way the vertices calculate the limiting values is as follows. Once a vertex receives an $\mbf{SB}$-pulse, it sends a pulse to its predecessors (in $\Nats$) announcing its currently displayed symbol (using the encoding $0\mapsto \mbf{B}_0$ and $1\mapsto \mbf{B}_1$). This is described in condition (vi). 

Once a vertex has received the collection of those symbols displayed on vertices greater than it, it asks whether both symbols $0$ and $1$ occur in this collection. If so, then it knows that its displayed symbol is not the limit, and so it enters state $u$ and does nothing (i.e., $u$ signifies termination for the subroutine). This is described in condition (viii). 

On the other hand, if a vertex sees that there is only one symbol displayed among those vertices larger than itself, then it knows that that symbol is the limiting value of the subroutine. The vertex then passes this information along to the activation vertex, via a pulse of type $\mbf{SF}_0$ or $\mbf{SF}_1$ (depending on the limiting value) and enters the subroutine termination state. This is described in condition (vii). 

Finally, if the activation vertex is in $a_2$ and receives exactly one pulse among $\mbf{SF}_0$ and $\mbf{SF}_1$, then  it knows that this is the limiting value, and sets its display symbol to that value and enters the subroutine termination state. This is described in condition (v). 

Of course, once a vertex is in the termination state, it will stay there, displaying the same symbol, unless and until it receives a pulse of type $\mbf{A}$. This is described in condition (ix). 

Condition (x) was added to complete the description of the lookup table, but 
because $h$ is total,
this will never occur in any actual run that begins at a starting configuration, even when this graph machine is embedded as a subroutine into a larger graph machine, as we subsequently describe.
Note that this subroutine will always complete its computation within $4$ timesteps.

We need one more graph machine to operate as a subroutine. The purpose of this subroutine will be to send pulses, in sequence, that activate other vertices.

Let $\G_{\Bmachine}^n$ be the graph satisfying the following.
\begin{itemize}
\item The underlying set is $\{\star_{-5n}, \dots, \star_0\}$.

\item There is only one label $q$, which all elements have. 

\item The colors are $\{\mbf{S}, \mbf{R}, \mbf{Q}, \mbf{A}, \mbf{SF}_0, \mbf{SF}_1\}$ 

\item The edge coloring is $E_{\Bmachine}^n$. 

\item The only vertex pairs having non-empty sets of edge colors are the following.

\begin{itemize}
\item $E_{\Bmachine}^n(\star_i,\star_{ i+1}) = \{\mbf{R}\}$
	for $-5n \leq i < 0$. 

\item $E_{\Bmachine}^n(\star_0, \star_0) = \{\mbf{S}\}$. 

\item $E_{\Bmachine}^n(\star_{0}, \star_{-5n}) = \{\mbf{Q}\}$.
\end{itemize}
\end{itemize}

We define the graph machine $\T_{\Bmachine}^n$ as follows. 

\begin{itemize}
\item The underlying graph is $\G_{\Bmachine}^n$.

\item The alphabet is $\{0,1\}$.
\item The states are $\{s, d, r, u\}$.

\item The lookup table $T_{\Bmachine}^n$ satisfies the following for all 
	$z$ in the alphabet, states $x$, and collections $X$ of colors. 
\begin{itemize}
\item[(i)] $T_{\Bmachine}^n(q, X, 0, s) = (\emptyset, 0, s)$ if $\{\mbf{R}, \mbf{S}\} \cap  X = \emptyset$.

\item[(ii)] $T_{\Bmachine}^n(q, X, 1, s) = (\{\mbf{S}\}, 0, s)$.

\item[(iii)] $T_{\Bmachine}^n(q, X \cup \{\mbf{S}\}, z, s) = (\{\mbf{Q}\}, 0, s)$ if $\mbf{R} \not \in X$.

\item[(iv)] $T_{\Bmachine}^n(q, X, z, s) = (\{\mbf{A}\}, 0, d)$ if $\{\mbf{Q},  \mbf{R}\} \cap X \neq \emptyset$.

\item[(v)] $T_{\Bmachine}^n(q, X, z, d) = (\{\mbf{R}\}, 0, r)$.

\item[(vi)] 
	$T_{\Bmachine}^n(q, X, z, r) = (\emptyset, k, d)$ if $\mbf{SF}_{k} \in X$ and $\mbf{SF}_{1- k} \not \in X$ for some $k\in\{0,1\}$, and otherwise
	$T_{\Bmachine}^n(q, X, z, r) = (\emptyset, 0, r)$.

\item[(vii)] $T_{\Bmachine}^n(q, X, z, u) = (X, z, u)$.
\item[(viii)] $T_{\Bmachine}^n(q, X, z, x) = (\emptyset, 0, u)$ in all other cases.
\end{itemize}
\end{itemize}

We now describe what the graph machine $\T_{\Bmachine}^n$ does. First notice that the only way for a vertex to get out of the initial state $s$ is for it to receive an $\mbf{S}$-pulse, a $\mbf{Q}$-pulse or an $\mbf{R}$-pulse. Only $\mbf{S}$-pulses can be sent from a vertex that is in the initial state and which hasn't received any other pulses. Also, there is only one $\mbf{S}$-edge, namely, a self loop at $\star_0$. Hence the first timestep of the graph machine's behavior is determined by what the vertex $\star_0$ initially displays. 

If the vertex $\star_0$ initially displays $0$, then all vertices display $0$ in the next step, and the subroutine does nothing else. This is described in condition (i).  If, however, vertex $\star_0$ initially displays $1$, then an $\mbf{S}$-pulse is sent by $\star_0$ to itself. This is described in condition (ii). Once $\star_0$ receives the $\mbf{S}$-pulse, it reverts back to displaying $0$, and sends a $\mbf{Q}$-pulse to $\star_{-5n}$. This is described in condition (iii). 

Once vertex $\star_{-5n}$ receives a $\mbf{Q}$-pulse, 
the main loop begins. In the main loop, first vertex $\star_{-5n}$ sends out an $\mbf{A}$-pulse and moves to a state $d$, as described in condition (iv). The purpose of the $\mbf{A}$-pulse is to tell the vertex that receives it to \emph{activate} and start calculating a limit. While there are no vertices in $\T_{\Bmachine}^n$ with $\mbf{A}$-colored edge, we will combine $\T_{\Bmachine}^n$ with copies of $\T_{\mfL}$, connecting the two graphs using 
$\mbf{A}$-colored edges. Note that only vertices of the form $\star_{-5k}$ with $k\le n$ will be connected to copies of $\T_{\mfL}$ via $\mbf{A}$-colored edges. The other vertices are there to provide additional timesteps in between $\mbf{A}$-pulses to allow the activated copies of $\T_{\mfL}$ time to complete their computations.

Once a vertex is in state $d$, it sends an $\mbf{R}$-pulse to its ``neighbor to the right'' (the vertex with least index greater than it), and moves to state $r$, where it will stay unless it is $\star_0$, as described in conditions (v) and (vi). Every vertex acts the same way upon arrival of an $\mbf{R}$-pulse as on an $\mbf{Q}$-pulse. Hence each vertex in the sequence sends, in succession, an $\mbf{A}$-pulse,  and then enters the state $r$. 

Condition (vii) ensures that when a vertex enters the subroutine termination state (which will only ever happens to $\star_0$ in the course of this subroutine's use by the larger program), the displayed symbol remains constant. Condition (viii) also describes a circumstance that happens only when the subroutine is used by the larger program, as does the first clause of condition (vi) (which agrees with condition (v) of $\T_\mfL$).

When we connect up $\T_{\Bmachine}^n$ with copies of $\T_{\mfL}$, vertex
$\star_0$ participates in calculating the final limit. 
Hence, in addition to having edges colored by $\mbf{A}$, 
vertex $\star_0$ also has $\mbf{SF}_0$ and $\mbf{SF}_1$-edges. If $\star_0$ receives a pulse of one of those colors, it then displays the corresponding value and moves to a state that keeps this value constant. 

We now combine the graph machines for subroutines, 
to get a graph machine that calculates $h$. 

For $e \in \Nats$, define the graph $\G_{g,e}$ as follows.
\begin{itemize}
\item The underlying set is $\Nats^n \cup \Nats^{n-1} \cup\cdots \cup  \Nats^2 \cup \Nats \cup \{\star_{-5n}, \dots, \star_0\}$. 

\item There are two labels, $p$ and $q$.  Vertex $\star_i$ is labeled by $q$ (for $-5n \leq i \leq 0$), and every other vertex is labeled by $p$. 

\item The colors are $\{\mbf{S}, \mbf{C}_0, \mbf{C}_1, \mbf{B}_0, \mbf{B}_1, \mbf{SB}, \mbf{SF}_0, \mbf{SF}_1, \mbf{A}, \mbf{Q}, \mbf{R}\}$.

\item The edge coloring is $E$, satisfying the following. 

	\begin{itemize}
\item For each $\a \in \Nats^{n-1}$ and $k, m \in \Nats$ with $k \neq m$ we have  the following.
\begin{itemize}
\item $E(\a k, \a m) = E_{\mfL}(k, m)$.

\item $E(\a k, \a k) = \{\mbf{C}_{g(\a k e)}, \mbf{B}_0, \mbf{B}_1\}$.

\item $E(\a, \a k) = E_{\mfL}(*, k)$ and $E(\a k, \a) = E_{\mfL}(k, *)$.

\item $E(\star_{-5n}, \a k) = \{\mbf{A}\}$. 

\end{itemize}

\item For each $\a \in \Nats^i$ for $0 < i < n-1$ and $k, m \in \Nats$ we have  the following.
\begin{itemize}
\item $E(\a k, \a m) = E_{\mfL}(k, m)$.

\item $E(\a, \a k) = E_{\mfL}(*, k)$ and $E(\a k, \a) = E_{\mfL}(k, *)$.

\item $E(\star_{-5(n-i+1)}, \a k) = \{\mbf{A}\}$. 
\end{itemize}

\item For each $k, m\in \Nats$ we have the following.
\begin{itemize}
\item $E(k, m) = E_{\mfL}(k, m)$.

\item $E(\star_0, k) = E_{\mfL}(*, k)$ and $E(k, \star_0) = E_{\mfL}(k, *)$.

\item $E(\star_0, \star_0) = \{\mbf{S}, \mbf{A}\}$. 
\end{itemize}

\item For each $-5n \leq k, m < 0$ we have the following.
\begin{itemize}
\item $E(\star_k, \star_m) = E_{\Bmachine}^n(\star_k, \star_m)$.

\item $E(\star_k, \star_0) = E_{\Bmachine}^n(\star_0, \star_0)$.
\end{itemize}
\end{itemize}
\end{itemize}

The graph $\G_{g,e}$ is such that for any tuple $\a \in \Nats^{\leq k}$, the set $\{\a\} \cup \{\a k\}_{k \in \Nats}$ is isomorphic to $\G_{\mfL}$ (after ignoring the edges $\{\mbf{C}_0, \mbf{C}_1\}$). This allows us to iteratively take limits. Further, each $\a \in \Nats^{n}$ has a self-loop which encodes the value of $g(\a e)$. This will be used to initialize the displayed symbols of vertices in the matrix that we  will later use to take the limits. 

We define the graph machine $\T_{g, e}$ as follows. 
\begin{itemize}
\item The underlying graph is $\G_{g, e}$.

\item The states are $\{d, s, a_0, a_1, a_2, u, b\}$.

\item The lookup table $T$ is such that the following hold, for all
	$z$ in the alphabet, states $t$, and collections $X$ of colors. 
\begin{itemize}
\item[(i)] $T(q, X, z, t) = T_{\Bmachine}^n(q, X, z, t)$. 

\item[(ii)] $T(p, X, z, t) = T_{\mfL}(p, X, z, t)$ if ($t \neq s$ and $\mbf{A} \not \in X$) or ($t = a_0$ and $\mbf{C}_k \in X$ and $\mbf{C}_{1-k} \not \in X$) for some $k \in \{0, 1\}$.

\item[(iii)] $T(p, X \cup \mbf{A}, z, s) = (\{\mbf{C}_0, \mbf{C}_1\}, 0, a_0)$.

	\item[(iv)] $T(p, X, z, a_0) = (\emptyset, k, u)$ if $\mbf{C}_k \in X$ and $\mbf{C}_{1-k} \not \in X$ for some $k\in\{0, 1\}$. 
\item[(v)] $T(p, X, z, t) = (\emptyset, 0, u)$ in all other cases.
\end{itemize}
\end{itemize}

We now describe a run of $\T_{g, e}$ on a starting configuration. First, just as with $\T_{\Bmachine}^n$, the computation begins by observing the behavior or $\star_0$. 
If $\star_0$ initially displays $0$, then all vertices stay in the initial state and display $0$ on the next timestep.
If $\star_0$ initially displays $1$, then the computation proceeds as in $\T_{\Bmachine}^n$, and 
vertex $\star_{-5n}$ sends an $\mbf{A}$-pulse, which \emph{activates} all vertices of the form $\a \in \Nats^n$. 

Vertices of the form $\a \in \Nats^n$ are not designed to compute the limits of anything, and so upon activation their values must be initialized. This is done by each such vertex attempting to send itself an $\mbf{C}_0$-pulse and $\mbf{C}_1$-pulse. In other words, each such vertex sends a $\mbf{C}_0$-pulse and $\mbf{C}_1$-pulse along all $\mbf{C}_0$-edges and
$\mbf{C}_1$-edges connected to it, respectively, if any exist (and all $\mbf{C}_0$-edges and $\mbf{C}_1$-edges connected to such a vertex are self-loops).
Because of how the graph $\G_{g,e}$ was constructed, each such $\a$ will only receive the pulse $\mbf{C}_{g(\a e)}$.
Hence after the pulse is received, vertex $\a$ completes its initialization by 
setting its displayed symbol to the index of whichever pulse it receives.

Meanwhile, 
vertices $\{\star_{-5n}, \dots, \star_0\}$ 
are sending $\mbf{R}$-pulses in succession along the sequence, causing each such vertex, in order,
to attempt to send an activation $\mbf{A}$-pulse (i.e., a pulse along all $\mbf{A}$-edges connected to it, of which there will be at most one). However, because only every fifth vertex in the sequence
$\{\star_{-5n}, \dots, \star_0\}$ is connected via an $\mbf{A}$-colored edge, by the time the next $\mbf{A}$-pulse is sent, i.e., along the edge attached to $\star_{-5n+5}$, the initialization procedure has finished. 

The next activation pulse is then sent from $\star_{-5n+5}$ to all vertices of the form $\a \in \Nats^{n-1}$. It causes these vertices to begin the process of calculating the limit of the sequence currently displayed by the vertices $\{\a 0, \a 1, \dots\}$. While this limit is being calculated, 
the vertices in 
$\{\star_{-5n+5}, \dots, \star_0\}$ are attempting to send out
activation pulses,
in sequence. By the time the next $\mbf{A}$-pulse is sent out, at $\star_{-5n+10}$, each vertex $\a\in\Nats^{n-1}$ is displaying the limit of the values displayed by $\{\a 0, \a 1, \dots\}$. 

This process then repeats until we get to $\star_0$, which plays a double role. First, once $\star_0$ has received an $\mbf{R}$-pulse, it sends out an $\mbf{A}$-pulse to itself. This signals $\star_0$ to begin the process of calculating the limit of the symbols displayed at $\{0, 1, \dots\}$. Secondly, when this calculations finishes, $\star_0$ displays 
\[
\lim_{x_0 \to \infty}\cdots \lim_{x_{n-1} \to \infty} g(x_0, \dots, x_{n-1}, e),\]
which is the desired value, $h(e)$. 

Note that the map $e \mapsto \T_{g,e}$ is computable, and that the lookup table $T$ is independent of $e$. Further, the time it takes for $\T_{g,e}$ to run is independent of $e$.  
Therefore $h$ is 
graph computable in constant time. 
\end{proof}

In summary, we define a ``subroutine'' graph machine that, on its own, computes the limit of a computable binary sequence. 
	We then embed $n$ repetitions of this subroutine into a single graph machine that computes the $n$-fold limit of the $(n+1)$-dimensional array given by $g$.
	The subroutine graph machine has a countably infinite sequence (with one special vertex) as its underlying graph, in that
	every vertex is connected to all previous vertices (and all are connected to the special vertex). Each vertex first activates itself, setting its displayed symbol to the appropriate term in the sequence whose limit is being computed. Each vertex sends a pulse to every previous vertex, signaling its displayed state. Any vertex which receives both $0$ and $1$ from vertices later in the sequence knows that the sequence alternates at some later index. 
Finally, any vertex which only receives a $0$ or $1$ pulse, but not both, sends a pulse corresponding to the one it receives to the special vertex. This special vertex then knows the limiting value of the sequence.

This technical construction allows us to conclude the following.

\begin{corollary}
\label{Pointwise computable Turing degrees}
Suppose $X\subseteq \Nats$ is such that $X \leT \mbf{0}^{(n)}$. Then 
$X$ is Turing-equivalent to some function that is 
graph computable in constant time by a machine whose underlying graph is finitary. 
\end{corollary}
\begin{proof}
	By Lemma~\ref{Description of arithmetical functions},
	$X$ is Turing equivalent to
	the $n$-fold limit of some computable function. By Proposition~\ref{Pointwise computable for iterative limits},
	this $n$-fold limit is graph computable in constant time
by a machine whose underlying graph is finitary. 
\end{proof}

Not only are all graph computable functions Turing reducible to $\mbf{0}^{(\w)}$, but this bound can be achieved.

\begin{theorem}
	There is a graph computable function 
	(via a machine whose underlying graph is finitary)
	that
	is Turing equivalent to $\mbf{0}^{(\w)}$.
	\label{omega-jump} 
\end{theorem}
\begin{proof}
	For each $n\in\Nats$, uniformly choose
	a computable function $g_n\:\Nats^{n+1}\to\Nats$ such that its $n$-fold limit $h_n$ satisfies
	$h_n \equivT \mbf{0}^{(n)}$. 
For $e,n \in \Nats$, let the graph machines $\G_{g_n, e}$ and $\T_{g_n, e}$
	be the graphs and graph machines described in the proof of Proposition~\ref{Pointwise computable for iterative limits}.
	
	Let $\G_\w$ be the graph which is the (computable) disjoint union of the uniformly computable graphs 
	$\{\G_{g_n, e}\st e, n \in \Nats\}$.
	Note that $\G_\w$ is finitary as the (finite) number of labels of $\G_{g_n, e}$ does not depend on $n$ or $e$.
	
	Further note that for all $e, n \in \Nats$, the graph machines $\T_{g_n, e}$ have the same lookup table. We can therefore let $\T_\w$ be the graph machine with underlying graph $\G_\w$ having this lookup table.
The function $\{\T_\w\}$ is total because each
$\{\T_{g_n, e}\}$ is  (and the corresponding submachines of the disjoint union do not interact).

	By the construction of $\T_\w$, we have
	$\mbf{0}^{(\w)} \leT \{\T_\w\}$. 
	On the other hand, $\{\T_\w\} \leT \mbf{0}^{(\w)}$ holds
	by Theorem~\ref{Total graph computable functions are computable from 0^w}.
	Hence $\{\T_\w\} \equivT \mbf{0}^{(\w)}$.
\end{proof}

\section{Finite degree graphs}
\label{finite-deg-sec}

We have seen that 
every arithmetical function is graph computable. However, as we will see in this section, if we instead limit ourselves to graphs where each vertex has finite degree, then not only is every graph computable function computable from $\mbf{0}'$, but also we can obtain more fine-grained control over the Turing degree of the function by studying the degree structure of the graph.

\subsection{Upper bound}

Before we move to the specific case of graphs of finite degree (defined below), there is an important general result concerning bounds on graph computability and approximations to computations. 

\begin{definition}
	Let $\Theta\: \Powerset_{<\w}(G) \to \Powerset_{<\w}(G)$.
We say that $\Theta$ is a \defn{uniform approximation} of $\T$ if
	for all finite subsets $A \subseteq G$,
\begin{itemize}
\item $A \subseteq \Theta(A)$, and

\item for any valid configuration $f$ for $\T$, the pair
	$(A, \Theta(A))$ is a $1$-approximation of $\T$ and $f$ on $A$.
\end{itemize}
\end{definition}

\begin{lemma}
\label{Uniform approximation implies approximation}
	Let $\Theta(A)$ be a uniform approximation of $\T$. Then for any finite subset $A$ of $G$, any valid configuration $f$ for $\T$, and any $n \in \Nats$, the tuple $(A, \Theta(A), \Theta^2(A), \dots, \Theta^n(A))$ is an $n$-approximation of $\T$ and $f$ on $A$.
\end{lemma}
\begin{proof}
We prove this by induction on $n$.

	\vspace*{5pt}
\noindent \ul{Base case:} $n = 1$ \nl
We know by hypothesis
	that $(A, \Theta(A))$ is a $1$-approximation of $\T$ and $f$ on $A$. 

	\vspace*{5pt}
\noindent \ul{Inductive case:} $n = k+1$\nl
We know that $(A, \Theta(A), \dots, \Theta^k(A))$ is a $k$-approximation of $\T$ and $f_1$ on $A$. It therefore suffices to show that $(\Theta^k(A), \Theta^{k+1}(A))$ is a $1$-approximation of $\T$ and $f$ on $A$. But this holds by our assumption on $\Theta$. 
\end{proof}

Note that while we will be able to get even better bounds in the case of finite degree graphs, we do have the following bound on computability. 
\begin{lemma}
\label{Computability properties of uniform approximations}
	Let $\Theta$ be a uniform approximation of $\T$. Then for any valid configuration $f$,
\begin{itemize}
\item[(a)] $\<\T, f\>$ is computable from $\Theta$ and $f$ (uniformly in $f$), and

\item[(b)] if $\{\T\}$ is total, then $\{\T\} \leT \Theta'$. 
\end{itemize}
\end{lemma}
\begin{proof}
Clause (a) follows from Lemma~\ref{Uniform approximation implies approximation} and the definition of an approximation.

	If $\{\T\}$ is total, then for each starting configuration $f$ of $\T$, there is an $n$ such that  $f_n = f_{n+1} 
	= \{\T\}(f)$. 
	Hence $\{\T\}$ is computable from the Turing jump of $\<\T, \pars\>$, and so it is computable from $\Theta'$. Therefore clause (b) holds.
\end{proof}

We now introduce the \emph{degree function} of a graph.

\begin{definition}
	For $v \in G$, define the \defn{degree} of $v$ to be the number of vertices incident with it, i.e.,
\[
\deg_\G(v)\defas |\{w\st E(v,w) \cup E(w,v) \neq \emptyset\}|,
\] 
and call $\deg_\G(\pars) \: G \to \Nats \cup \{\infty\}$ the \defn{degree function} of $\G$.

We say that $\G$ has \defn{finite degree} when 
	$\mathrm{rng}(\deg_\G) \subseteq \Nats$,
and say that
$\G$ has \defn{constant degree} when $\deg_\G$ is constant.
\end{definition}

We will see that for a graph $\G$ of finite degree, its degree function 
bounds the computability of $\G$-computable functions.

The following easy lemma will allow us to provide a computation bound on graph Turing machines all vertices of whose underlying graph have finite degree.
\begin{lemma}
\label{Degree computable from 0'}
Suppose  that
	$\G$ has finite degree. Then $\deg_\G \leT \mbf{0}'$. 
\end{lemma}
\begin{proof}
	Let $G$ be the underlying set of $\G$.
	Because $\G$ is computable, for any vertex $v\in G$, the set of neighbors of $v$ is computably enumerable, uniformly in $v$. The size of this set is therefore $\mbf{0}'$-computable, uniformly in $v$, and so $\deg_\G \leT \mbf{0}'$.
\end{proof}

Recall the definition of $n$-neighborhood (\cref{nneigh}).

\begin{lemma}
\label{Properties of 1-neighborhoods}
Suppose  that
$\G$ 
	has finite degree. Then 
\begin{itemize}
\item[(a)] the $1$-neighborhood map $\neighbor_1$ is computable from $\deg_\G$, and

\item[(b)] for any $\G$-machine $\T$, the map $\neighbor_1$ is a uniform approximation to $\T$.  

\end{itemize}

\end{lemma}
\begin{proof}
	Clause (a) follows from the fact that given the degree of a vertex one can search for all of its neighbors, as this set is computably enumerable (uniformly in the vertex) and of a known finite size. 

Clause (b) follows from the fact that if a vertex receives a pulse, it must have come from some element of its $1$-neighborhood. 
\end{proof}

We now obtain the following more precise upper bound on complexity for finite degree graphs.

\begin{theorem}
\label{Computability of Finite Degree Graph}
Suppose that $\G$ has finite degree and 
	$\{\T\}$ is total. Then $\{\T\}$ is computable from $\deg_\G$ and its range is contained in $\Lang^{<G}$. 
\end{theorem}
\begin{proof}
	First note that if $f$ is a starting configuration of $\T$ and for all $k' > k$ we have  $f(k') = 0$, then for any $m \in \Nats$ and any $v \in G \setminus \neighbor_m(\{0, \dots, k\})$, we have that $\<\T, f\>(v, m) = \<\T, f\>(v, 0)$. Therefore $\{\T\}(f)$ halts in $m$ steps if and only if $\{\T|_{\neighbor_m(\{0, \dots, k\})}\}(f|_{\neighbor_m(\{0, \dots, k\})})$ halts in $m$ steps. 
	Therefore we can determine whether or not $\{\T\}$ halts in $m$ steps by examining $\{\T|_{\neighbor_m(\{0, \dots, k\})}\}(f|_{\neighbor_m     (\{0, \dots, k\})})$, which is uniformly computable from $\neighbor_m$ by \cref{lemma-finite-graph}, since $\neighbor_m(\{0, \dots, k\})$ is finite.

But $\neighbor_m$ is just $\neighbor_1^m$, and by Lemma~\ref{Properties of 1-neighborhoods}(a), $\neighbor_1$ is computable from $\deg_\G$. 
	For each $m$ we can uniformly $\deg_\G$-computably check whether $\{\T\}(f)$ halts at stage $m$.
	Hence $\{\T\}$ is $\deg_\G$-computable as
	$\{\T\}$ halts on all starting configurations. 
\end{proof}

We then obtain the following important corollaries.

\begin{corollary}
\label{Finite degree implies computable in 0'}
Suppose that $\G$ 
	has finite degree.
Then any $\G$-computable function is $\mbf{0}'$-computable.
\end{corollary}
\begin{proof}
	This follows from Theorem~\ref{Computability of Finite Degree Graph} and Lemma~\ref{Degree computable from 0'}.
\end{proof}

\begin{corollary}
\label{Constant degree implies computable}
Suppose that $\G$ 
has
	constant degree.
Then any $\G$-computable function is computable (in the ordinary sense). 
\end{corollary}
\begin{proof}
	This follows from Theorem~\ref{Computability of Finite Degree Graph} and the fact that $\deg_\G$ is constant (and hence computable).
\end{proof}

\subsection{Lower bound}
\label{finite-degree-lower-bounds}

In this subsection, we consider the possible Turing degrees of graph computable functions where the underlying graph has finite degree. In particular, we show that every Turing degree below $\mbf{0}'$ is the degree of some total graph computable function where the underlying graph has finite degree. 

Recall from Lemma~\ref{Description of arithmetical functions} (for $n = 1$) that a set $X \subseteq \Nats$ satisfies $X \leT \mbf{0}'$ when the characteristic function of $X$ is the limit of a 2-parameter computable function.
The following standard definition (see \cite[III.3.8]{MR882921})
describes a hierarchy among sets $X \leT \mbf{0}'$.

\begin{definition}
Let $A\subseteq \Nats$ be such that $A \leT \mbf{0}'$, and let $h$ be the characteristic function of $A$.
	For $k \in \Nats$, the set $A$ is
\defn{$k$-computably enumerable},
or $k$-c.e., when there is a computable function $g\:\Nats \times \Nats \to \{0,1\}$ such that 
\begin{itemize}
\item $(\forall m \in \Nats)\ g(0, m) = 0$, 
\item $(\forall m \in \Nats)\ h(m) = \lim_{n \to \infty}g(n, m)$, and

\item for each $m \in \Nats$, $|\{\ell\in \Nats\st g(\ell,m) \neq g(\ell+1, m)\}| \le k$, i.e., the uniformly computable approximation to the limit alternates at most $k$-many times for each input. 
\end{itemize}
\end{definition}

In particular, the $1$-c.e.\ sets are precisely the c.e.\ sets. 

We next construct a collection of graph machines $\{\M_e\st e\in\Nats\}$ that we will use in \cref{Lower bound on finite degree main result}.

\begin{lemma}
\label{Lower bound piece finite degree}
There is a finite lookup table $\F$ and a collection, definable uniformly in $e$, of graph machines $\M_e$ having edge coloring $E_e$ and common lookup table $\F$ such that whenever $e \in \Nats$ satisfies
\begin{itemize}
\item $\{e\} \: \Nats \to \{0, 1\}$ is total, and

\item $\lim_{n \to \infty}\{e\}(n) = m_e$ exists,
\end{itemize}
then the following hold.
\begin{itemize}
\item $\G_e$, the underlying graph of $\M_e$, has only finitely many edges and each vertex is of degree at most 3.

\item If $f$ is a valid configuration of $\M_e$ where all vertices are in the initial state, then 
\begin{itemize}
	\item if $f_{[1]}(0) = 0$, i.e., vertex $0$ displays $0$ in the configuration $f$,
	then $\<\M_e,f\>$ halts with every vertex displaying $0$, and
\item if $f_{[1]}(1) = 0$, i.e., vertex $0$ displays $1$ in the configuration $f$, then $\<\M_e,f\>$ halts with vertex $0$ displaying $\lim_{n \to \infty}\{e\}(n) = m_e$ (and every other vertex displaying $0$). 
\end{itemize}
\item If $|\{n\st \{e\}(n) \neq \{e\}(n+1)\}| < k$, then $\M_e$ halts on any starting configuration in at most
	$(2k+4)$-many timesteps. 
\end{itemize}
\end{lemma}
\begin{proof}
	First we describe the graphs $\G_e$. For notational convenience, write $\{e\}(-1) \defas 1 - \{e\}(0)$. The edges are determined by the following, for $n \ge 0$.
\begin{itemize}
\item $E_e(0, 1) = \{r\}$,  and $E_e(2,0) = \{b\}$, and $E_e(1,2) = \{g\}$.  

\item If $\{e\}(n) = \{e\}(n+1)$, then vertices $2n+1$ and $2n+2$ have degree $0$, i.e., there are no edges connecting them to any other vertices. 

\item If $\{e\}(n) \neq \{e\}(n+1)$, then $E_e(2k+3, 2n+3) = E_e(2n+3, 2n+4) = E_e(2n+4, 2k+4) = \{g\}$, where $k$ is the ``most recent alternation'', i.e., the largest $k$ such that $-1 \le k < n$ and $\{e\}(k) \neq \{e\}(k+1)$.
\end{itemize}

We now define the lookup table $\F$. 
\begin{itemize}
	\item There are two states: $s$ (the initial state) and $a$.

\item If a vertex displays $1$, then it sends an $r$-pulse and sets its display to $0$. 

\item If a vertex receives an $r$-pulse or a $g$-pulse, then it sends a $b$-pulse and a $g$-pulse. 

\item If a vertex receives a $b$-pulse, then it sets its state to $a$ and alternates its displayed symbol. 
\end{itemize}

When $\M_e$ is run, on the first step every vertex sets its displayed symbol to $0$. If vertex $0$
had initially displayed $0$, then this is all that happens. However, if vertex $0$ had initially displayed $1$, then vertex $0$ will also send out a $r$-pulse to vertex $1$. This is the only $r$-pulse that will ever be sent (as $0$ is the only vertex which is the source of an $r$-colored edge, and it will not send any other $r$-pulses). 

At the second stage, this $r$-pulse can be thought of as being converted to a $g$-pulse. Further, $g$-pulses have the property that (1) they propagate forward along directed edges, splitting whenever possible, and (2) they turn into a $b$-pulse when being sent from vertex $2$ to vertex $0$. As such, the number of $b$-pulses that vertex $0$ receives is the number of vertices in $\M_e$ with two out edges (i.e., the number of times a $g$-pulse \emph{splits} in two). However, by construction, the number of such vertices is the number of times that $\{e\}$ alternates values. 

	Hence the number of $b$-pulses that vertex $0$ receives is equal to the number of times that $\{e\}$ alternates values. But vertex $0$ alternates the symbols it displays exactly when it receives a $b$-pulse, and so when the graph reaches a halting configuration (in the sense that the subsequent configuration one timestep later is the same),
	vertex $0$ will display $m_e$. 
\end{proof}

\begin{theorem}
\label{Lower bound on finite degree main result}
For every $X\: \Nats \to \{0,1\}$ such that $X \leT \mbf{0}'$ there is a graph machine $\N_X$ with lookup table $\F$ and finitary underlying graph $\H_X$ such that 
\begin{itemize}
\item every vertex of $\H_X$ has degree at most $3$,

\item $\{\N_X\}$ is total and Turing equivalent to $X$, and

\item if $X$ is $k$-c.e.\ 
	then $\{\N_X\}$ halts in $(2k+4)$-many steps on any input. 
\end{itemize}
In particular, $\N_X$ runs in linear space.
\end{theorem}

\begin{proof}
	Let $Y$ be a computable function such that $X(m) = \lim_{n \to \infty}Y(n, m)$ for all $m\in\Nats$,  and such that
	if $X$ is $k$-c.e.,
	then $Y$ witnesses this fact. Let $y_m$ be a code such 
	$\{y_m\}(n) = Y(n,m)$ for all $n\in\Nats$.

Recall the graphs
	$\G_{e}$ (for $e\in\Nats$)
	defined in Lemma~\ref{Lower bound piece finite degree},
	and let $\H_X$ be the disjoint union of 
	$\{\G_{y_m}\st m\in\Nats\}$
	in the following sense:
	\begin{itemize}
		\item	$\H_X$ is a graph with underlying set $\Nats\times\Nats$.
\item 	For each $m \in \Nats$, let $\H_{X,m}$ be the subgraph of $\H_X$ consisting of elements whose second coordinate is $m$. Then
	the map 
defined by 
$(n, m) \mapsto n$  for $n\in\Nats$
			is an isomorphism from $\H_{X,m}$
	to $\G_{y_m}$.
\item 	There are no edges between any elements of $\H_{X,m}$ and $\H_{X, m'}$ 
	for distinct $m, m' \in\Nats$.
	\end{itemize}
	Note that $\H_X$ is finitary (and in fact has just one label) by the construction of the graphs $\G_{e}$, which have common lookup table $\F$.

	Suppose that $x \in \{0, 1\}^{\Nats \times \Nats}$ has only finitely many $0$'s. As earlier, let $\widehat{x}$ be the valid configuration satisfying 
$\widehat{x}(\ell) = (\emptyset, x(\ell), s)$ for all $\ell \in \Nats \times \Nats$.
	Then by Lemma~\ref{Lower bound piece finite degree}, $\<\N_X, \widehat{x}\>$ halts and satisfies the following.
\begin{itemize}
	\item For any $n, m \in \Nats$, the vertex
	$(n+1, m)$ displays $0$. 

\item For any $m \in \Nats$, if $x(0, m) = 0$ then $\<\N_X, \widehat{x}\>$ halts with $0$ displayed at vertex $(0, m)$. 

\item For any $m \in \Nats$, if $x(0, m) = 1$ then $\<\N_X, \widehat{x}\>$ halts with $X(m)$ displayed at vertex $(0, m)$. 
\end{itemize}

In particular, $\{\N_X\}$ is a total function that is Turing equivalent to $X$. 
	Finally note that 
	given any connected subgraph $A$ on which the starting configuration is supported, the computation is completely determined by the $(2k+4)$-neighborhood of that subgraph.
	Further, each vertex has degree at most $3$, and so the size of each such neighborhood is at most $3^{2k+4} \cdot |A|$.
	Hence $\N_X$ runs in linear space.  
\end{proof}

\section{Graph-theoretic properties that do not affect computability}
\label{Other graph-theoretic properties}

We now show that the graph machines themselves can be taken to be efficiently computable without affecting the Turing degrees of the functions that can be computed via them.
We also show that, while we have made use of directed edges throughout the above constructions, 
from a computability perspective this was just a notational convenience.
The arguments in this section will be sketched, and we refer the reader to a forthcoming extended version of this paper for detailed discussion.

\subsection{Efficiently computable graphs}
\label{eff-comp-graph-subsec}
So far we have considered how resource bounds interact with graph computability, though the graph machines themselves have remained arbitrary computable structures.
Here we show that requiring the graph machines to be efficiently computable has no 
effect on the complexity functions that can be graph computed.

Cenzer and Remmel \cite{MR1130218} showed that an arbitrary computable relational structure is computably isomorphic to some polynomial-time structure. We base our result here on their key ideas.

		\begin{definition}
			A graph machine $\T$ is \defn{polynomial-time computable} if (a) its underlying 
			graph $\G$ is polynomial-time computable, in the sense that
			set $G$ is a polynomial-time computable subset of $\Nats$, the 
			labels $L$  and colors $C$ 
			are polynomial-time computable subsets of $\Nats$, and the functions $V$, $E$, and $\gamma$ are polynomial-time computable, and (b) the functions $S$, $\alpha$, and $T$ are polynomial-time computable, all with respect to a standard encoding of finite powersets.
		\end{definition}
		
		\begin{proposition}
			For every computable graph machine $\T$, there is a computably isomorphic graph machine $\T^\circ$ that polynomial-time computable.
			\label{eff-comp}
		\end{proposition}

		\begin{proofsketch}
We will build the polynomial-time computable graph machine      $\T^\circ$ along with its computable isomorphism to $\T$ simultaneously.
		Suppose $\{a_i \st i\in\Nats\}$ is a computable increasing enumeration of $G$.
		We will define $\{b_i \st i\in\Nats\}$ by induction so that the map $a_i \mapsto b_i$ is the desired isomorphism.

		Suppose we have defined $\{b_i \st i < n\}$. First 
		consider if any new labels, colors, or states are needed to describe the relationship $a_n$ and $\{a_i \st i < n\}$.  If there are such labels, colors, or states,  we look at the transition table $T$ and consider the length of a minimal computations needed to calculate $T$ on these new labels, colors, and states, along with the ones previously identified. 
		Note that there can be at most finitely many new labels, colors, and states introduced at this stage, and so there is an upper bound on the length of these computations.
		We then add corresponding new labels, colors, or states to the structure $\T^\circ$ where the natural numbers associated to these labels, colors, or states are large enough to ensure that the computation of $T^\circ$ remains polynomial-time computable.

		Similarly, we consider the length of the computation needed to determine
		$V(a_n)$, $E(a_i, a_n)$, $E(a_n, a_i)$, and $\gamma(a_n)$ for $i< n$ and choose a natural number to assign to $b_n$ which is large enough to ensure  that
		the corresponding
		$V^\circ(b_n)$, $E^\circ(b_i, b_n)$, $E^\circ(b_n, b_i)$, and $\gamma^\circ(b_n)$ for $i< n$  are all polynomial-time computable.
		One can verify that this yields the desired polynomial-time computable graph machine and computable isomorphism.
		\end{proofsketch}

		In particular, because computably isomorphic graph machines compute functions of the same Turing degrees, 
		in
\cref{Pointwise computable Turing degrees}
and
	\cref{omega-jump} 
		we may obtain the stated graph computable functions via graph machines whose underlying graphs are polynomial-time computable.
Likewise, in \cref{Lower bound on finite degree main result}, we may take the underlying graph $\H_X$ to be polynomial-time computable.

\subsection{Symmetric graphs}
\label{symmetric-graphs}

\begin{definition}
	The graph $\G$
	is \defn{symmetric} if 
	$E(v,w) = E(w,v)$
	for all $v, w \in G$.
\end{definition}

\begin{proposition}
	Every graph computable function is $\G_S$-comput\-able for some symmetric graph $\G_S$.
			\label{symmetric-ok}
\end{proposition}
		\begin{proofsketch}
	Given a computable graph $\G$ and a $\G$-machine $\T$,
define $\G_S$ to be the \emph{symmetric} graph with underlying set $G_S = G \cup \{e_{(v,w)}, e^*_{(v, w)} \st v, w\in G \}$, set of colors $C$, and edge coloring $E_S$ defined such that
	
\begin{align*}
	E(v,w) &=
	E_S(v, e_{(v,w)}) = E_S(e_{(v,w)}, e^*_{(v, w)}) \\
	&= E_S(e^*_{(v,w)}, w) = E_S(v, e^*_{(v,w)})
\end{align*}
for all $v, w \in G$.

In particular, whenever $c \in E(v,w)$ we have that (i) $e^*_{(v,w)}$ is connected to $w$ via an edge of color $c$, and (ii) there are paths of length 1 and length 2 (of color $c$) connecting $v$ and $e^*_{(v, w)}$. 

We now define $\T_S$ to be the $\G_S$-machine that simulates $\T$ as follows. Every three timesteps of $\T_S$ corresponds to one timestep of $\T$. State transitions in $\T_S$ between elements of $G$ are the same as in $\T$ (except that each transition takes three timesteps to be processed). However, any time that 
	a vertex $v$ sends a $c$-pulse to $w$
	within $\T$, 
	the corresponding pulse in $\T_S$ instead goes from $v$ to both $e_{(v,w)}$ and $e^*_{(v,w)}$. When $e_{(v,w)}$ receives a $c$-pulse, it sends a $c$-pulse to $e^*_{(v,w)}$. However, $e^*_{(v,w)}$ only sends a $c$-pulse to $w$ after it has itself received $c$-pulses in \emph{two} successive timesteps (i.e., first one from $v$, and then one from $e_{(v,w)}$). 

In this way, in $\T_S$ when $v$ sends out a $c$-pulse, this causes (three timesteps later) a $c$-pulse to reach all vertices $w \in G$ for which $c\in E(v,w)$. However, even though in $\G_S$ the relation $E_S$ is symmetric, when $v$ sends out a $c$ pulse it will not reach any $w \in G$ such that $E(w,v)$ holds. 
	Hence the behavior of $\T_S$ within $G$ simulates the (slowed-down) behavior of $\T$.
		\end{proofsketch}

\section{Representations of other computational models via graph machines}
\label{representations-sec}

We now describe several other models of computation, mainly on graphs, and describe how they can be viewed as special cases of graph machines. These examples provide further evidence for graph machines being a universal model of computation on graphs.

\subsection{Ordinary Turing machines}
\label{otm-subsec}

We begin by showing how to simulate an ordinary Turing machine by a graph Turing machine.

Let $M$ be a Turing machine (with a finite set of states and finite transition function, as usual).
Then there is an equivalent graph machine whose underlying graph 
represents a Turing machine tape, as we now describe.
The underlying graph is a one-sided chain with underlying set $\{-1\} \cup \Nats$.
Every vertex in $\Nats$ has degree $2$ (with edges both to and from its predecessor on the left and successor on the right), while $-1$ is connected only to $0$.
There are two labels, one of which holds of all of $\Nats$ and other of which holds at $-1$. 

Any vertex of $\Nats$ which is in the initial state and has not received any pulse does nothing. When the vertex $-1$ is in the initial state and displays $1$, it sends a pulse to $0$ to signal the creation of the read/write head ``above $0$''.
The presence of a head above a vertex is encoded via the state of that vertex. 
The lookup table of the graph machine will ensure that there is a unique head, above precisely one vertex, at each timestep following any starting configuration with $1$ displayed at $-1$.

At any subsequent timestep, only the vertex with the head will send out pulses.
Between two any adjacent vertices in $\Nats$, there are sufficiently many edge colors to transmit the current state of $M$, and so at each timestep, the vertex with the head 
uses the transition function of $M$ to determine the location of the head at the next time step, and (if the head needs to move according to $M$) transmits the appropriate signal (containing the current state of $M$ as well as signaling that the head is above it now), to its neighbor, and accordingly adjusts its own state. If the head does not need to move, the vertex with the head merely sets its own state and displayed symbol according to the transition function of $M$.

One can show that not only does this embedding yield a graph machine that computes functions of the same Turing degree as
the original Turing machine $M$, but that the output of this graph machine
produces (on $\Nats$) the exact same
function as $M$, moreover via a weak bisimulation (in which the function is computed with 
merely a small linear time overhead).

The doubly-infinite one-dimensional read/write tape of an ordinary Turing machine has cells indexed by $\Z$, the free group on one generator, and in each timestep the head moves according to the generator or its inverse. This interpretation of a Turing machine as a $\Z$-machine has been generalized to $H$-machines for arbitrary finitely generated groups $H$ by \cite{Aubrun2016}, and our simulation above extends straightforwardly to this setting as well.

One might next consider how cleanly one might embed various extensions of Turing machines where the tape is replaced by a graph, such as Kolmogorov--Uspensky machines \cite{KU}, Knuth's pointer machines \cite[pp.\ 462--463]{MR0286317}, and Sch\"onhage's storage modification machines \cite{MR584506}. For a further discussion of these and their relation to sequential abstract state machines, see \cite{gurevich1993kolmogorov} and \cite{MR1858823}.

\subsection{Cellular automata}
\label{cellular-subsec}
We now consider cellular automata; for background,
see, e.g., the book \cite{toffoli1987cellular}.
Cellular automata (which we take to be always
finite-dimensional, finite-radius
and with finitely-many states) can be naturally simulated by graph machines, moreover of constant degree. In particular, 
\cref{Constant degree implies computable}
applies to this embedding.

Not only is the evolution of every such cellular automata computable (moreover via this embedding as a graph machine), but there are particular automata whose evolution encodes the behavior of a universal Turing machine
	(\cite{MR2211290} and \cite{MR2493991}).
Several researchers have also considered the possibility of expressing intermediate Turing degrees via this evolution (\cite{BaldwinReview}, \cite{CohnReview}, and \cite{MR1964811}). Analogously, one might ask which Turing degrees can be expressed in the evolution of graph machines.

We now describe this embedding.
Cells of the automata are taken to be the vertices of the graph, and cells are connected to its ``neighbors'' (other cells within the given radius)
by a collection of edges (of the graph) whose labels encode their relative position (e.g., ``1 to the left of'') and all possible cellular automaton states.  (In particular, every vertex has the same finite degree.)
The displayed symbol of each vertex encodes the cellular automaton state of that cell.

The rule of the cellular automaton is encoded in the lookup table so as to achieve the following:
At the beginning of each timestep, each cell announces its state by virtue of its vertex sending 
out a pulse to the vertices of the neighboring cells, along edges whose labels encode this state.  
(If during some timestep a vertex does not receive a pulse from the direction corresponding to a neighboring cell, then it assumes that the vertex corresponding to the neighboring cell is in its initial state and is displaying $0$.)
Each vertex then updates its displayed symbol based on how the corresponding cell should update its state, based on the states of its neighbors, according to the rule of the original cellular automaton.

Note that this encoding produces a bisimulation between the original cellular automaton and the graph machine built based on it.

\subsection{Parallel graph dynamical systems}
\label{gds-subsec}

Parallel graph dynamical systems \cite{MR3332130} can be viewed as essentially equivalent to the finite case of graph Turing machines, as we now describe.
Finite cellular automata 
can also be viewed as a special case of parallel graph dynamical systems, 
as can finite boolean networks,
as noted in \cite[\S2.2]{MR3332130}.
For more on parallel graph dynamical systems, see \cite{MR3339901}, \cite{MR3332130}, and \cite{MR2078971}.

A parallel graph dynamical system specifies a finite (possibly directed) graph and an evolution operator (determining how a given tuple of states for the vertices of the graph transitions into a new tuple). The evolution operator is required to be local, in the sense of determining the new state of a vertex given only its state and those of its adjacencies.

We may embed a parallel graph dynamical system as a graph machine similarly to how we embedded cellular automata above. 
However, because different vertices have non-isomorphic neighborhoods, we can no longer label edges connecting a vertex to a given neighbor based on the ``direction'' of this neighbor. We therefore require a different collection of label types for every vertex of the graph, which are used to signal the source of the edge.

This embedding produces a bisimulation between the given parallel graph dynamical system and the graph machine built from it.
Note that this embedding also works for the natural infinite extension of parallel graph dynamical systems, where each vertex is required to have finite in-degree. This restriction on in-degree ensures that each vertex of the corresponding graph machine has only finitely many edge colors, even though the set of all edge colors may be infinite.

In contrast, it is not immediately clear how best to encode an arbitrary (parallel) abstract state machine \cite{MR2002271} as a graph Turing machine (due to the higher arity relations of the ASM).

\section{Possible extensions}
\label{Possible-extensions}

We now describe several possible extensions of graph Turing machines that may be worth studying.

It would be interesting to develop a non-deterministic version of graph Turing machines, where certain pulses are allowed, but not required, to be sent.
		We expect that such a framework would naturally encompass an infinitary version of Petri nets \cite{DBLP:journals/csur/Peterson77}, as well as the machines described by
\cite{MR3163227}.

	Another interesting extension would be a randomized version of graph Turing machines, where the pulses fire independently according to a specified probability distribution. In this setting, one could then study the joint distribution of overall behavior. Randomized graph Turing machines might encompass dynamical Bayesian networks \cite{MR2704368} and other notions of probabilistic computation on a graph.

	Our framework involves underlying graphs that are specified before the computation occurs. Is there a natural way to extend our framework to allow for new nodes to be added, destroyed, duplicated, or merged?  Such an extension might naturally encompass infinitary generalizations of 
models of concurrency and parallelism based on graph rewrite rules, such as
		interaction nets \cite{DBLP:conf/popl/Lafont90} and bigraphs \cite{DBLP:books/daglib/0022395}.

\section{Open questions}
\label{Open Questions}

We conclude with several open questions.

\begin{itemize}

	\item Does every Turing degree below $\mbf{0}^{(\w)}$ contain a graph computable function?  So far, we merely know that every arithmetical degree contains a graph computable function ---  but there are Turing degrees below $\mbf{0}^{(\w)}$ that are not below $\mbf{0}^{(n)}$ for any $n\in\Nats$.

	\item Are there graph computable functions that are not graph computable in constant time?  In particular, can $\mbf{0}^{(\w)}$ be graph computed in constant time? (The construction in \cref{omega-jump} is linear-time.)

	\item Are there graph computable functions that are not graph computable by a graph machine with finitary underlying graph? 

\end{itemize}

\section*{Acknowledgements}
The authors would like to thank Tomislav Petrovi\'c, Linda Brown Westrick, and the anonymous referees of earlier versions for helpful comments. 

\vspace*{20pt}

\bibliographystyle{amsnomr}
\bibliography{bibliography}

\end{document}